\documentclass[11pt, twoside]{article}
\usepackage{mathrsfs}
\usepackage{amsfonts}
\usepackage{authblk}
\usepackage{amsmath,amssymb}
\usepackage{multicol}
\usepackage{multicol,graphics}
\usepackage{graphics}
\usepackage{epstopdf}
\usepackage[dvips]{graphicx}
\usepackage{latexsym}
\usepackage{mathrsfs}
\usepackage{amsmath}
\usepackage[all]{xy}
\usepackage{graphicx}
\usepackage{amssymb}
\usepackage{color}
\usepackage{amsthm}
\usepackage{txfonts}
\usepackage{bbm}
\usepackage{tikz}
\newtheorem{theorem}{Theorem}[section]
\newtheorem{prop}[theorem]{Proposition}
\theoremstyle{definition}
\newtheorem{defn}[theorem]{Definition}
\newtheorem{lemma}[theorem]{Lemma}
\newtheorem{coro}[theorem]{Corollary}
\newtheorem{prop-def}{Proposition-Definition}[section]
\newtheorem{coro-def}{Corollary-Definition}[section]

\newtheorem{remark}[theorem]{Remark}

\newtheorem{exam}[theorem]{Example}

\newcommand{\nc}{\newcommand}
\nc{\tred}[1]{\textcolor{red}{#1}}
\nc{\tblue}[1]{\textcolor{blue}{#1}}
\nc{\tgreen}[1]{\textcolor{green}{#1}}
\nc{\tpurple}[1]{\textcolor{purple}{#1}}
\nc{\btred}[1]{\textcolor{red}{\bf #1}}
\nc{\btblue}[1]{\textcolor{blue}{\bf #1}}
\nc{\btgreen}[1]{\textcolor{green}{\bf #1}}
\nc{\btpurple}[1]{\textcolor{purple}{\bf #1}}
\nc{\NN}{{\mathbb N}}
\nc{\ncsha}{{\mbox{\cyr X}^{\mathrm NC}}} \nc{\ncshao}{{\mbox{\cyrX}^{\mathrm NC}_0}}

\newcommand{\efootnote}[1]{}

\renewcommand{\textbf}[1]{}

\newcommand{\delete}[1]{}

\nc{\mlabel}[1]{\label{#1}}  
\nc{\mcite}[1]{\cite{#1}}  
\nc{\mref}[1]{\ref{#1}}  
\nc{\mbibitem}[1]{\bibitem{#1}} 

\delete{
\nc{\mlabel}[1]{\label{#1}{\hfill \hspace{1cm}{\bf{{\ }\hfill(#1)}}}}
\nc{\mcite}[1]{\cite{#1}{{\bf{{\ }(#1)}}}}  
\nc{\mref}[1]{\ref{#1}{{\bf{{\ }(#1)}}}}  
\nc{\mbibitem}[1]{\bibitem[\bf #1]{#1}} 
}

\newcommand{\tdun}[1]{\begin{picture}(10,5)(-2,-1)
\put(0,0){\circle*{2}}
\put(3,-2){\tiny #1}
\end{picture}}

\newcommand{\tduns}[1]{\begin{picture}(10,5)(-2,-1)
\put(0,0){\circle*{2}}
\put(3,-2){\tiny $\sigmaup$}
\end{picture}}

\newcommand{\tddeuxs}[2]{\begin{picture}(12,5)(0,-1)
\put(3,0){\circle*{2}}
\put(3,0){\line(0,1){5}}
\put(3,5){\circle*{2}}
\put(6,-2){\tiny $\sigmaup$}
\put(6,3){\tiny $\sigmaup$}
\end{picture}}

\newcommand{\tddeuxx}[2]{\begin{picture}(12,5)(0,-1)
\put(3,0){\circle*{2}}
\put(3,0){\line(0,1){5}}
\put(3,5){\circle*{2}}
\put(6,-2){\tiny #1}
\put(6,3){\tiny #2}
\end{picture}}

\newcommand{\tdtroisuns}[3]{\begin{picture}(20,12)(-5,-1)
\put(3,0){\circle*{2}}
\put(-0.65,0){$\vee$}
\put(6,7){\circle*{2}}
\put(0,7){\circle*{2}}
\put(5,-2){\tiny $\sigmaup$}
\put(9,5){\tiny $\sigmaup$}
\put(-5,5){\tiny $\sigmaup$}
\end{picture}}

\newcommand{\tdtroisdeux}[3]{\begin{picture}(12,12)(-2,-1)
\put(0,0){\circle*{2}}
\put(0,0){\line(0,1){5}}
\put(0,5){\circle*{2}}
\put(0,5){\line(0,1){5}}
\put(0,10){\circle*{2}}
\put(3,-2){\tiny #1}
\put(3,3){\tiny #2}
\put(3,9){\tiny #3}
\end{picture}}

\newcommand{\squares}[4]{\begin{picture}(12,5)(0,-1)
\put(3,0){\circle*{2}}
\put(8,0){\circle*{2}}
\put(8,5){\circle*{2}}
\put(3,5){\circle*{2}}
\put(3,0){\line(0,1){5}}
\put(3,0){\line(1,0){5}}
\put(8,0){\line(0,1){5}}
\put(3,5){\line(1,0){5}}
\put(-2,-2){\tiny #1}
\put(-2,4){\tiny #2}
\put(11,-2){\tiny #3}
\put(11,4){\tiny #4}
\end{picture}}
\newcommand{\squarea}[4]{\begin{picture}(12,5)(0,-1)
\put(3,0){\circle*{2}}
\put(8,0){\circle*{2}}
\put(8,5){\circle*{2}}
\put(3,5){\circle*{2}}
\put(3,0){\line(0,1){5}}
\put(3,0){\line(1,0){5}}
\put(8,0){\line(0,1){5}}
\put(-2,-2){\tiny #1}
\put(-2,4){\tiny #2}
\put(11,-2){\tiny #3}
\put(11,4){\tiny #4}
\end{picture}}
\newcommand{\squareb}[4]{\begin{picture}(12,5)(0,-1)
\put(3,0){\circle*{2}}
\put(8,0){\circle*{2}}
\put(8,5){\circle*{2}}
\put(3,5){\circle*{2}}
\put(3,0){\line(0,1){5}}
\put(3,0){\line(1,0){5}}
\put(3,5){\line(1,0){5}}
\put(-2,-2){\tiny #1}
\put(-2,4){\tiny #2}
\put(11,-2){\tiny #3}
\put(11,4){\tiny #4}
\end{picture}}
\newcommand{\squarec}[4]{\begin{picture}(12,5)(0,-1)
\put(3,0){\circle*{2}}
\put(8,0){\circle*{2}}
\put(8,5){\circle*{2}}
\put(3,5){\circle*{2}}
\put(3,0){\line(0,1){5}}
\put(8,0){\line(0,1){5}}
\put(3,5){\line(1,0){5}}
\put(-2,-2){\tiny #1}
\put(-2,4){\tiny #2}
\put(11,-2){\tiny #3}
\put(11,4){\tiny #4}
\end{picture}}
\newcommand{\squared}[4]{\begin{picture}(12,5)(0,-1)
\put(3,0){\circle*{2}}
\put(8,0){\circle*{2}}
\put(8,5){\circle*{2}}
\put(3,5){\circle*{2}}
\put(3,0){\line(1,0){5}}
\put(8,0){\line(0,1){5}}
\put(3,5){\line(1,0){5}}
\put(-2,-2){\tiny #1}
\put(-2,4){\tiny #2}
\put(11,-2){\tiny #3}
\put(11,4){\tiny #4}
\end{picture}}
\newcommand{\para}[2]{\begin{picture}(12,5)(0,-1)
\put(3,0){\circle*{2}}
\put(8,0){\circle*{2}}
\put(3,0){\line(1,0){5}}
\put(-2,-2){\tiny #1}
\put(11,-2){\tiny #2}
\end{picture}}

\newcommand{\trpoina}[3]{\begin{picture}(12,5)(0,-1)
\put(3,0){\circle*{2}}
\put(8,0){\circle*{2}}
\put(8,5){\circle*{2}}
\put(3,0){\line(1,0){5}}
\put(8,0){\line(0,1){5}}
\put(-2,-2){\tiny #1}
\put(11,-2){\tiny #2}
\put(11,4){\tiny #3}
\end{picture}}

\newcommand{\trpoinb}[3]{\begin{picture}(12,5)(0,-1)
\put(3,0){\circle*{2}}
\put(8,0){\circle*{2}}
\put(3,5){\circle*{2}}
\put(3,0){\line(1,0){5}}
\put(3,0){\line(0,1){5}}
\put(-2,-2){\tiny #1}
\put(-2,4){\tiny #2}
\put(11,-2){\tiny #3}
\end{picture}}

\newcommand{\trpoinc}[3]{\begin{picture}(12,5)(0,-1)
\put(3,0){\circle*{2}}
\put(8,5){\circle*{2}}
\put(3,5){\circle*{2}}
\put(3,0){\line(0,1){5}}
\put(3,5){\line(1,0){5}}
\put(-2,-2){\tiny #1}
\put(-2,4){\tiny #2}
\put(11,4){\tiny #3}
\end{picture}}

\newcommand{\trpoind}[3]{\begin{picture}(12,5)(0,-1)
\put(8,0){\circle*{2}}
\put(8,5){\circle*{2}}
\put(3,5){\circle*{2}}
\put(8,0){\line(0,1){5}}
\put(3,5){\line(1,0){5}}
\put(-2,4){\tiny #1}
\put(11,-2){\tiny #2}
\put(11,4){\tiny #3}
\end{picture}}

\nc{\opa}{\ast} \nc{\opb}{\odot} \nc{\op}{\bullet} \nc{\pa}{\frakL}
\nc{\arr}{\rightarrow} \nc{\lu}[1]{(#1)} \nc{\mult}{\mrm{mult}}
\nc{\diff}{\mathfrak{Diff}}
\nc{\opc}{\sharp}\nc{\opd}{\natural}
\nc{\ope}{\circ}
\nc{\dpt}{\mathrm{d}}
\nc{\hck}{H_{RT}}
\nc{\vdf}{\calf}
\nc{\ldf}{\calf_\ell}
\nc{\hlf}{H}
\nc{\onek}{\mathbf{1}_\bfk}
\nc{\diam}{alternating\xspace}
\nc{\Diam}{Alternating\xspace}
\nc{\cdiam}{canonical alternating\xspace}
\nc{\Cdiam}{Canonical alternating\xspace}
\nc{\AW}{\mathcal{A}}

\nc{\ari}{\mathrm{ar}}

\nc{\lef}{\mathrm{lef}}

\nc{\Sh}{\mathrm{ST}}

\nc{\Cr}{\mathrm{Cr}}

\nc{\st}{{Schr\"oder tree}\xspace}
\nc{\sts}{{Schr\"oder trees}\xspace}

\nc{\vertset}{\Omega} 
\nc{\pb}{{\mathrm{pb}}}
\nc{\Lf}{{\mathrm{Lf}}}
\nc{\lft}{{left tree}\xspace}
\nc{\lfts}{{left trees}\xspace}
\nc{\fat}{{fundamental averaging tree}\xspace}
\nc{\fats}{{fundamental averaging trees}\xspace}
\nc{\avt}{\mathrm{Avt}}
\nc{\rass}{{\mathit{RAss}}}
\nc{\aass}{{\mathit{AAss}}}
\nc{\vin}{{\mathrm Vin}}    
\nc{\lin}{{\mathrm Lin}}    
\nc{\inv}{\mathrm{I}n}
\nc{\gensp}{V} 
\nc{\genbas}{\mathcal{V}} 
\nc{\bvp}{V_P}     
\nc{\gop}{{\,\omega\,}}     

\nc{\bin}[2]{ (_{\stackrel{\scs{#1}}{\scs{#2}}})}  
\nc{\binc}[2]{ \left (\!\! \begin{array}{c} \scs{#1}\\
    \scs{#2} \end{array}\!\! \right )}  
\nc{\bincc}[2]{  \left ( {\scs{#1} \atop
    \vspace{-1cm}\scs{#2}} \right )}  
\nc{\bs}{\bar{S}} \nc{\cosum}{\sqsubset} \nc{\la}{\longrightarrow}
\nc{\rar}{\rightarrow} \nc{\dar}{\downarrow} \nc{\dprod}{**}
\nc{\dap}[1]{\downarrow \rlap{$\scriptstyle{#1}$}}
\nc{\md}{\mathrm{dth}} \nc{\uap}[1]{\uparrow
\rlap{$\scriptstyle{#1}$}} \nc{\defeq}{\stackrel{\rm def}{=}}
\nc{\disp}[1]{\displaystyle{#1}} \nc{\dotcup}{\
\displaystyle{\bigcup^\bullet}\ } \nc{\gzeta}{\bar{\zeta}}
\nc{\hcm}{\ \hat{,}\ } \nc{\hts}{\hat{\otimes}}
\nc{\barot}{{\otimes}} \nc{\free}[1]{\bar{#1}}
\nc{\uni}[1]{\tilde{#1}} \nc{\hcirc}{\hat{\circ}} \nc{\lleft}{[}
\nc{\lright}{]} \nc{\lc}{\lfloor} \nc{\rc}{\rfloor}
\nc{\curlyl}{\left \{ \begin{array}{c} {} \\ {} \end{array}
    \right .  \!\!\!\!\!\!\!}
\nc{\curlyr}{ \!\!\!\!\!\!\!
    \left . \begin{array}{c} {} \\ {} \end{array}
    \right \} }
\nc{\longmid}{\left | \begin{array}{c} {} \\ {} \end{array}
    \right . \!\!\!\!\!\!\!}
\nc{\onetree}{\bullet} \nc{\ora}[1]{\stackrel{#1}{\rar}}
\nc{\ola}[1]{\stackrel{#1}{\la}}
\nc{\ot}{\otimes} \nc{\mot}{{{\boxtimes\,}}}
\nc{\otm}{\overline{\boxtimes}} \nc{\sprod}{\bullet}
\nc{\scs}[1]{\scriptstyle{#1}} \nc{\mrm}[1]{{\rm #1}}
\nc{\margin}[1]{\marginpar{\rm #1}}   
\nc{\dirlim}{\displaystyle{\lim_{\longrightarrow}}\,}
\nc{\invlim}{\displaystyle{\lim_{\longleftarrow}}\,}
\nc{\mvp}{\vspace{0.3cm}} \nc{\tk}{^{(k)}} \nc{\tp}{^\prime}
\nc{\ttp}{^{\prime\prime}} \nc{\svp}{\vspace{2cm}}
\nc{\vp}{\vspace{8cm}} \nc{\proofbegin}{\noindent{\bf Proof: }}
\nc{\proofend}{$\blacksquare$ \vspace{0.3cm}}
\nc{\modg}[1]{\!<\!\!{#1}\!\!>}
\nc{\intg}[1]{F_C(#1)} \nc{\lmodg}{\!
<\!\!} \nc{\rmodg}{\!\!>\!}
\nc{\cpi}{\widehat{\Pi}}
\nc{\sha}{{\mbox{\cyr X}}}  
\nc{\shap}{{\mbox{\cyrs X}}} 
\nc{\shpr}{\diamond}    
\nc{\shp}{\ast} \nc{\shplus}{\shpr^+}
\nc{\shprc}{\shpr_c}    
\nc{\msh}{\ast} \nc{\zprod}{m_0} \nc{\oprod}{m_1}
\nc{\vep}{\varepsilon} \nc{\labs}{\mid\!} \nc{\rabs}{\!\mid}
\nc{\sqmon}[1]{\langle #1\rangle}

\nc{\mmbox}[1]{\mbox{\ #1\ }} \nc{\dep}{\mrm{dep}} \nc{\fp}{\mrm{FP}}
\nc{\rchar}{\mrm{char}} \nc{\End}{\mrm{End}} \nc{\Fil}{\mrm{Fil}}
\nc{\Mor}{Mor\xspace} \nc{\gmzvs}{gMZV\xspace}
\nc{\gmzv}{gMZV\xspace} \nc{\mzv}{MZV\xspace}
\nc{\mzvs}{MZVs\xspace} \nc{\Hom}{\mrm{Hom}} \nc{\id}{\mrm{id}}
\nc{\im}{\mrm{im}} \nc{\incl}{\mrm{incl}} \nc{\map}{\mrm{Map}}
\nc{\mchar}{\rm char} \nc{\nz}{\rm NZ} \nc{\supp}{\mathrm Supp}

\nc{\Alg}{\mathbf{Alg}} \nc{\Bax}{\mathbf{Bax}} \nc{\bff}{\mathbf f}
\nc{\bfk}{{\bf k}} \nc{\bfone}{{\bf 1}} \nc{\bfx}{\mathbf x}
\nc{\bfy}{\mathbf y}
\nc{\base}[1]{\bfone^{\otimes ({#1}+1)}} 
\nc{\Cat}{\mathbf{Cat}}

\nc{\detail}{\marginpar{\bf More detail}
    \noindent{\bf Need more detail!}
    \svp}
\nc{\Int}{\mathbf{Int}} \nc{\Mon}{\mathbf{Mon}}
\nc{\rbtm}{{shuffle }} \nc{\rbto}{{Rota-Baxter }}
\nc{\remarks}{\noindent{\bf Remarks: }} \nc{\Rings}{\mathbf{Rings}}
\nc{\Sets}{\mathbf{Sets}} \nc{\wtot}{\widetilde{\odot}}
\nc{\wast}{\widetilde{\ast}} \nc{\bodot}{\bar{\odot}}
\nc{\bast}{\bar{\ast}} \nc{\hodot}[1]{\odot^{#1}}
\nc{\hast}[1]{\ast^{#1}} \nc{\mal}{\mathcal{O}}
\nc{\tet}{\tilde{\ast}} \nc{\teot}{\tilde{\odot}}
\nc{\oex}{\overline{x}} \nc{\oey}{\overline{y}}
\nc{\oez}{\overline{z}} \nc{\oef}{\overline{f}}
\nc{\oea}{\overline{a}} \nc{\oeb}{\overline{b}}
\nc{\weast}[1]{\widetilde{\ast}^{#1}}
\nc{\weodot}[1]{\widetilde{\odot}^{#1}} \nc{\hstar}[1]{\star^{#1}}
\nc{\lae}{\langle} \nc{\rae}{\rangle}
\nc{\lf}{\lfloor}
\nc{\rf}{\rfloor}
\nc{\indc}[2]{#1[#2]}

\nc{\QQ}{{\mathbb Q}}
\nc{\RR}{{\mathbb R}} \nc{\ZZ}{{\mathbb Z}}

\nc{\cala}{{\mathcal A}} \nc{\calb}{{\mathcal B}}
\nc{\calc}{{\mathcal C}}
\nc{\cald}{{\mathcal D}} \nc{\cale}{{\mathcal E}}
\nc{\calf}{{\mathcal F}} \nc{\calg}{{\mathcal G}}
\nc{\calh}{{\mathcal H}} \nc{\cali}{{\mathcal I}}
\nc{\call}{{\mathcal L}} \nc{\calm}{{\mathcal M}}
\nc{\caln}{{\mathcal N}} \nc{\calo}{{\mathcal O}}
\nc{\calp}{{\mathcal P}} \nc{\calr}{{\mathcal R}}
\nc{\cals}{{\mathcal S}} \nc{\calt}{{\mathcal T}}
\nc{\calu}{{\mathcal U}} \nc{\calw}{{\mathcal W}} \nc{\calk}{{\mathcal K}}
\nc{\calx}{{\mathcal X}} \nc{\CA}{\mathcal{A}}

\nc{\fraka}{{\mathfrak a}} \nc{\frakA}{{\mathfrak A}}
\nc{\frakb}{{\mathfrak b}} \nc{\frakB}{{\mathfrak B}}
\nc{\frakD}{{\mathfrak D}} \nc{\frakF}{\mathfrak{F}}
\nc{\frakf}{{\mathfrak f}} \nc{\frakg}{{\mathfrak g}}
\nc{\frakH}{{\mathfrak H}} \nc{\frakL}{{\mathfrak L}}
\nc{\frakM}{{\mathfrak M}} \nc{\bfrakM}{\overline{\frakM}}
\nc{\frakm}{{\mathfrak m}} \nc{\frakP}{{\mathfrak P}}
\nc{\frakN}{{\mathfrak N}} \nc{\frakp}{{\mathfrak p}}
\nc{\frakS}{{\mathfrak S}} \nc{\frakT}{\mathfrak{T}}
\nc{\frakX}{{\mathfrak X}}
\nc{\BS}{\mathbb{S
}}

\font\cyr=wncyr10 \font\cyrs=wncyr7
\nc{\xing}[1]{\textcolor{red}{Xing:#1}}
\nc{\meng}[1]{\textcolor{blue}{xiaomeng: #1}}
\nc{\revise}[1]{\textcolor{red}{#1}}

\nc{\ID}{{\rm I}}\nc{\lbar}[1]{\overline{#1}}\nc{\bre}{{\rm bre}}
\nc{\sd}{\cals}\nc{\rb}{\rm RB}\nc{\A}{\rm A}\nc{\LL}{\rm L}\nc{\tx}{\tilde{X}}
\nc{\col}{\Delta_{G}}\nc{\mul}{m_{G}}\nc{\epl}{\varepsilon_{G}}
\nc{\cog}{\Delta_{\epsilon}}
\nc{\coc}{\Delta_{Cay}}\nc{\muc}{m_{Cay}}\nc{\uc}{u_{Cay}}\nc{\epc}{\varepsilon_{Cay}}
\nc{\hl}{H_{G}}\nc{\arro}[1]{#1}\nc{\px}{P_{\tx}}\nc{\pw}{P_{\mathfrak{w}}}\nc{\pl}{B^+}
\nc{\pp}{\pl}\nc{\ppp}[1]{B^+(#1)}\nc{\dw}{\diamond_{\mathfrak{w}}}\nc{\dl}{\diamond_{\rm \ell}}
\nc{\ncshaw}{\sha^{{\rm NC}}_{\mathfrak{w}}}\nc{\ncshal}{\sha^{{\rm NC}}_{{\rm \ell}}}
\nc{\ver}{\rm V}\nc{\ld}{l}\nc{\del}{\Delta_{{\rm \ell}}}\nc{\epsl}{\varepsilon_{{\rm \ell}}}
\nc{\uul}{u_{{\rm \ell}}}\nc{\oneh}{\mathbf{1}}\nc{\onew}{\mathbf{1}}
\nc{\etree}{\mathbbm{1}} \nc{\conc}{m_{G}} \nc{\medmid}{{\,~{\tiny \longmid}~\,}}
\nc{\leql}{\leq_{\text{l}}} \nc{\leqh}{\leq_{\text{h}}}
\nc{\leqhl}{\leq_{\text{h,l}}}  \nc{\lhl}{<_{\text{h,l}}} \nc{\Cay}{\rm C\,}
\nc{\graphc}{\mathfrak{G}}\nc{\hopfc}{\mathfrak{H}} \nc{\ul}{\etree} \nc{\sg}{\mathfrak{S}}

\begin{document}
\title{A Hopf algebra on subgraphs of a graph}
\author{Xiaomeng Wang$^\dagger$, Shou-Jun Xu, Xing Gao}

\renewcommand{\thefootnote}{\fnsymbol{footnote}}



\footnotetext {E-mail address: wangxm2015@lzu.edu.cn}

\affil[]{\small School of Mathematics and Statistics, Gansu Key Laboratory of Applied \authorcr Mathematics  and Complex Systems, Lanzhou University, Lanzhou, Gansu 730000,  China}

\date{}

\maketitle
\begin{abstract}
In this paper, we construct a bialgebraic and further a Hopf algebraic structure on top of subgraphs of a given graph.
Further, we give the dual structure of this Hopf algebraic structure.
We study the algebra morphisms induced by graph  homomorphisms, and obtain a covariant functor from a graph category to an algebra category.

\vspace{0.5cm}
{\bf Keywords}: Hopf algebra, covariant functor, graph category, algebra category
\end{abstract}


\section{Introduction}

The Hopf algebraic theory has a long history and broad applications in mathematics and physics~\cite{CK, Fo0, Gal, KC, Kr2, MM, Mo, Swb, ZGG}.
The Hopf algebra of rooted forests originated from the work of A. D\"{u}r~\cite{DA} and had a very rich theory.
For example, A. Connes and D. Kreimer introduced the Connes-Kreimer Hopf algebra of rooted forests in~\cite{CK,Kr} to study the renormalization of perturbative
quantum field theory. In~\cite{CK, CK1, CK2}, they introduced commutative Hopf algebras of graphs to study the combinatorial structure of renormalization.
It is also related to the Loday-Ronco Hopf algebra~\cite{LR, ZG18} and Grossman-Larson Hopf algebra~\cite{GL} of rooted trees. L. Foissy  has done a lot of work on the Hopf algebra of rooted forests~\cite{Fo0, Fo1, Fo2, Fo3}.
As related results, various infinitesimal bialgebras  on decorated rooted forests have
been established in~\cite{Gw,GZ, ZCGL18}, via different Hochschild 1-cocycle conditions.

A totally assigned graph is a graph with a total order on its edges.
In~\cite{Dh}, G. E. Duchamp et al. built a Hopf algebra of totally assigned graphs.
L. Foissy~\cite{Fo16} constructed a Hopf algebra on all graphs to insert chromatic polynomial into the theory of combinatorial Hopf algebra,
and gave a new proof of some classical results.

The concept of graph (incidence) Hopf algebras was introduced by Schmitt~\cite{Sch}. Namely, a graph algebra is a
commutative, cocommutative, graded, connected Hopf algebra, whose basis elements correspond to finite graphs.
Humpert and Martin~\cite{Hm} obtained a new nonrecursive formula for the antipode of a graph (incidence) Hopf algebra.
In~\cite{Lan}, the author introduced the $4$-bialgebra of graphs that satisfies some relations.

The study of Hopf algebras of graphs is extensive, but the research on Hopf algebras
of subgraphs has no result yet. Subgraph plays an important role in the study of graph theory.
For example, we can study the prefect graph by the independence number and clique number of the induced subgraph.
We can determine if a graph is an interval map by the induced subgraph and many more.

In the present paper, we construct a Hopf algebraic structure on subgraphs of a given graph.
In particular, a combinatorial description of the coproduct is given.
Narrowing ourself to the algebraic part of the aforementioned Hopf algebra on subgraphs of a given graph, we
obtain an algebra morphism from a graph homomorphism. Using the language of categories, we obtain a covariant functor
from the category of graphs to the category of algebras.

Here is the structure of the paper.
In Section~\mref{sec:CKHOPHAL}, we first review the definitions of free monoid and free module.
Then we proceed to give an algebraic structure on subgraphs of a graph. 
 Next, we construct a coalgebraic structure on subgraphs of a graph (Lemma~\mref{le:lecounicoal}), thereby we construct a bialgebra on the subgraph (Theorem~\mref{th:bialg}).
Finally, to make the coproduct $\col$ more explicit, we describe it in a combinatorial way (Eq.~(\mref{eq:comni})).

In Section~\mref{sec:hopfalg}, continuing the line in Section~\mref{sec:CKHOPHAL}, a Hopf algebra structure on subgraphs of a graph is given (Theorem~\mref{th:thcg}).
Further, we study the dual Hopf algebra structure. We conclude with a description of the dual Hopf algebra (Theorem~\mref{le:ledup}).
This section is also devoted to an algebra morphism induced by a graph homomorphism (Theorem~\mref{th:grhop}).
As an application, we obtain a functor from graph category to algebra category~(Theorem~\mref{th:thcf}).

{\bf Notation.} Graphs considered in this paper are connected and undirected graphs without multiple edges and loops. We will be working over a unitary commutative base ring $\bfk$.
Linear maps and tensor products are taken over $\bfk$.
For an algebra $A$, we view $A\ot A$ as an $A$-bimodule via
\begin{equation}
a\cdot(b\otimes c):=ab\otimes c\,\text{ and }\, (b\otimes c)\cdot a:= b\otimes ca.
\mlabel{eq:dota}
\end{equation}

\section{A bialgebra on subgraphs of a given graph}
\mlabel{sec:CKHOPHAL}
In this section, we build a bialgebraic structure on top of subgraphs of a given graph.

\subsection{An algebraic structure on subgraphs of a given graph}\mlabel{subs:sub1}
Let us first review some notations on graphs and algebras, which are used throughout the remainder of the paper.
For a graph $G$, denote by $V(G)$ and $E(G)$ the vertex set and edge set of $G$, respectively.

\begin{defn}~\cite{Dies} Let $G$ be a graph.
\begin{enumerate}
\item If its $n$ vertices are distinguished from one another by labels such as $\{1, 2\cdots, n\}$, then $G$ is said to be {\bf labeled}.

\item If there is a path between any two vertices of $G$, then $G$ is {\bf connected}.

\item A {\bf subgraph} of a graph $G$ is simply a graph, all of whose vertices belong to $V(G)$ and all of whose edges belong to $E(G)$.

\item For $S\subseteq V(G)$, the {\bf induced subgraph} $G[S]$ in $G$ by $S$ is the graph whose edge set consists of all of the edges in $G$ that have both endpoints in $S$.
\end{enumerate}
\end{defn}

The concepts of free monoid and free module are needed later.

\begin{defn}~\cite{Hob}
\begin{enumerate}
\item A {\bf semigroup} is a nonempty set $S$ together with a binary operation $\cdot: S\times S\to S$ which is associative:
$$(x\cdot y)\cdot z =  x\cdot(y\cdot z)\,\text{ for all }\, x, y, z\in S.$$
A semigroup $S$  is called a {\bf monoid} if it contains an element $1$ with the property that
$$x\cdot 1=1\cdot x=x\,\text{ for all }\, x\in S.$$
\item
Let $(S,\cdot_S,1_S)$ and $(T,\cdot_T,1_T)$ be monoids. A map $\phi: S\rightarrow T$ is called a {\bf monoid morphism} if
$$\phi(x\cdot_Sy)=\phi(x)\cdot_T\phi(y)\,\,\text{ and }\,\,\phi(1_S)=1_T\,\text{ for all }\, x, y\in S.$$

\item
The {\bf free monoid } on a set $X$ is a monoid $S$ together with a map $i: X\rightarrow S$ with the property that, for any monoid $T$ and map $\phi: X\rightarrow T$, there exists a unique monoid morphism $\psi: S\rightarrow T$ such that $\psi\circ i=\phi$.
\end{enumerate}
\end{defn}

\begin{defn}~\cite{Hun}
\begin{enumerate}
\item Let $\bfk$ be a ring. A (left) {\bf $\bfk$-module} is an additive abelian group $M$ together with a function $\bfk\times M\rightarrow M$, $(k,x)\mapsto kx$ such that for all $k,l\in \bfk$ and $x,y\in M$:
    \begin{enumerate}
    \item $k(x+y)=kx+ky;$
    \item $(k+l)x=kx+lx;$
    \item $k(lx)=(kl)x.$
    \end{enumerate}

\item Let $M$ and $N$ be modules over ring $\bfk$. A map $f: M\rightarrow N$ is a {\bf $\bfk$-module morphism} if
$$f(x+y)=f(x) + f(y)\,\,\text{ and }\,\,f(kx)=kf(x) \,\text{ for all }\, x, y\in M,\,\, k\in \bfk.$$

\item
The {\bf free $\bfk$-module} on a set $X$ is a $\bfk$-module $F$ together with a set map $i: X\rightarrow F$ with
the property that, for any $\bfk$-module $M$ and map $f: X\rightarrow M$, there exists a unique $\bfk$-module morphism $\bar{f}: F\rightarrow M$ such that $\bar{f} \circ i=f$.
\end{enumerate}
\end{defn}

A $\bfk$-module $M$ is free if and only if it has a $\bfk$-linear basis~\cite{Hun}.

{\it In the rest of this paper,} let $G$ be a graph with vertices labelled by natural numbers and let $\calg$ be the set of nonempty connected subgraphs of $G$.
Let $M(\calg)=\langle\calg\rangle$ be the free monoid on $\calg$ in which the multiplication is the concatenation, denoted by $\mul$ and usually suppressed.
The unit in $M(\calg)$ is the empty graph, denote by $\etree$.
An element $F$ in $M(\calg)$ is a noncommutative product of connected subgraphs in $\calg$:
$$F = \Gamma_1\cdots \Gamma_n, \,\text{ where } n\geq 0\,\text{ and }\, \Gamma_1, \cdots, \Gamma_n\in \calg.$$
Here we employ the convention that $F = \etree$ whenever $n=0$. We define $\bre(F):=n$ to be the {\bf breadth} of $F$.
For example, $$\bre(\etree)=0,\,\bre(\tdun{1})=1,\,\bre(\tddeuxx{$\tiny 1$}{$\tiny 2$}\tddeuxx{$\tiny 1$}{$\tiny 2$}\tdun{1}\tdun{2})=4.$$
Denote by $$\hlf(\calg):= \bfk M(\calg)$$ the free $\bfk$-module spanned by $M(\calg)$.
By linearity, the multiplication $\mul$ can be extended to $\hlf(\calg)$ and then
the triple $(H(\calg),\,\mul,\,\ul)$ is an algebra. Indeed, it is the free algebra on the set $\calg$.

\begin{exam}
Here are some examples of $\hlf(\calg)$.
\begin{enumerate}
\item If $G=\bullet_1$, then $\calg=\{\bullet_1\}$, $ M(\calg)=\langle   \bullet_1\rangle$ and $\hlf(\calg)=\bfk\langle  \bullet_1\rangle$.

\item If $G=\tddeuxx{$\tiny 1$}{$\tiny 2$}$, then
\begin{align*}
\calg=\{\tdun{1},\tdun{2},\tddeuxx{$\tiny 1$}{$\tiny 2$}\},\, M(\calg)=\langle  \tdun{1},\tdun{2},\tddeuxx{$\tiny 1$}{$\tiny 2$}\rangle\,\text{ and }\,\hlf(\calg)=\bfk\langle  \tdun{1},\tdun{2},\tddeuxx{$\tiny 1$}{$\tiny 2$}\rangle.
\end{align*}

\item If $G=\squares{$2$}{$1$}{$3$}{$4$}$\,\,, then
\begin{align*}
\calg=& \ \{\tdun{1},\tdun{2},\tdun{3},\tdun{4},\tddeuxx{$\tiny 2$}{$\tiny 1$},\tddeuxx{$\tiny 3$}{$\tiny 4$},\para{$1$}{$4$}\,\,,\,\,\para{$2$}{$3$}\,\,,\,\,\trpoina{$2$}{$3$}{$4$}\,\,,\,\,\trpoinb{$2$}{$1$}{$3$}\,\,,\,\,\trpoinc{$2$}{$1$}{$4$}\,\,,\,\,
\trpoind{$1$}{$3$}{$4$}\,\,,\,\,\squarea{$2$}{$1$}{$3$}{$4$}\,\,,\,\,\squareb{$2$}{$1$}{$3$}{$4$}\,\,,
\,\,\squarec{$2$}{$1$}{$3$}{$4$}\,\,,\,\,\squared{$2$}{$1$}{$3$}{$4$}\,\,,\,\,\squares{$2$}{$1$}{$3$}{$4$}\,\,\},\\
M(\calg)=&\ \langle  \tdun{1},\tdun{2},\tdun{3},\tdun{4},\tddeuxx{$\tiny 2$}{$\tiny 1$},\tddeuxx{$\tiny 3$}{$\tiny 4$},\para{$1$}{$4$}\,\,,\,\,\para{$2$}{$3$}\,\,,\,\,\trpoina{$2$}{$3$}{$4$}\,\,,\,\,\trpoinb{$2$}{$1$}{$3$}\,\,,\,\,\trpoinc{$2$}{$1$}{$4$}\,\,,\,\,
\trpoind{$1$}{$3$}{$4$}\,\,,\,\,\squarea{$2$}{$1$}{$3$}{$4$}\,\,,\,\,\squareb{$2$}{$1$}{$3$}{$4$}\,\,,
\,\,\squarec{$2$}{$1$}{$3$}{$4$}\,\,,\,\,\squared{$2$}{$1$}{$3$}{$4$}\,\,,\,\,\squares{$2$}{$1$}{$3$}{$4$}\,\,\rangle,\\
\hlf(\calg)=& \bfk\langle  \tdun{1},\tdun{2},\tdun{3},\tdun{4},\tddeuxx{$\tiny 2$}{$\tiny 1$},\tddeuxx{$\tiny 3$}{$\tiny 4$},\para{$1$}{$4$}\,\,,\,\,\para{$2$}{$3$}\,\,,\,\,\trpoina{$2$}{$3$}{$4$}\,\,,\,\,\trpoinb{$2$}{$1$}{$3$}\,\,,\,\,\trpoinc{$2$}{$1$}{$4$}\,\,,\,\,
\trpoind{$1$}{$3$}{$4$}\,\,,\,\,\squarea{$2$}{$1$}{$3$}{$4$}\,\,,\,\,\squareb{$2$}{$1$}{$3$}{$4$}\,\,,
\,\,\squarec{$2$}{$1$}{$3$}{$4$}\,\,,\,\,\squared{$2$}{$1$}{$3$}{$4$}\,\,,\,\,\squares{$2$}{$1$}{$3$}{$4$}\,\,\rangle.
\end{align*}
\mlabel{ex:itex3}
\end{enumerate}
\mlabel{ex:exs}
\end{exam}

\subsection{A coproduct on subgraphs of a given graph}\mlabel{subs:sub2}
In this subsection, we construct a coproduct on subgraphs of a graph.
Let us first recall

\begin{defn}~\cite[Definition 2.1]{MM}
A {\bf coalgebra} $(C,\,\Delta,\,\varepsilon)$ over $\bfk$ is a $\bfk$-module $C$ together with morphisms of $\bfk$-modules
$\Delta:= C\rightarrow C\ot C$, called the {\bf coproduct}, and $\varepsilon:= C\rightarrow\bfk$, called the {\bf counit},
such that the  diagrams
$$\xymatrix{
  C \ar[d]_{\Delta} \ar[r]^{\Delta} & C\ot C \ar[d]^{\id_C\ot\Delta} \\
  C\ot C\ar[r]^{\Delta\ot\id_C} & C\ot C\ot C   }$$
and
$$\xymatrix{
  \bfk\ot C
                & C\ot C \ar[l]_{\varepsilon\ot\id_C}\ar[r]^{\id_C\ot\varepsilon}&C\ot\bfk  \\
                & C  \ar[ur]_{\beta_r}\ar[u]^{\Delta}\ar[ul]^{\beta_l}  }$$
are commutative. Here $\beta_l$ and  $\beta_r$ are the isomorphisms defined by
\begin{align*}
&\beta_l:C\rightarrow\bfk\ot C,\,\,x \mapsto 1_\bfk\ot x,\\
&\beta_r:C\rightarrow C\ot \bfk,\,\,x \mapsto x\ot 1_\bfk\,\,\text{ for } x\in C.
\end{align*}
If a coalgebra $(C,\,\Delta,\,\varepsilon)$ satisfies $\tau\circ\Delta=\Delta$, where $\tau$ is a $\bfk$-linear map$$\tau:C\ot C\rightarrow C\ot C,\,\,x\ot y\rightarrow y\ot x\,\,\text{ for }\,\,x, y\in C,$$
 then the coalgebra $(C,\,\Delta,\,\varepsilon)$ is said to be a {\bf cocomutative coalgebra}.
\mlabel{de:decoa}
\end{defn}
\begin{defn}~\cite[p. 54]{Abe}
Let $(C,\,\Delta,\,\varepsilon)$ be a coalgebra. A submodule $D$ of a coalgebra $C$ is called a {\bf $\bfk$-subcoalgebra} of $C$ if $D$ is a coalgebra with the restriction of $\Delta$ and $\varepsilon$.
Further, $D$ is called a  {\bf simple $\bfk$-subcoalgebra} if it does not have any $\bfk$-subcoalgebras other than $\{0\}$ and $D$.
\end{defn}

Now we define a  coproduct $\col$ on $\hlf(\calg)$.
By linearity, we only need to define $\col(F)$ for basis elements $F\in M(\calg)$.
If $\bre(F)=0$, then $F=\etree$ and define
\begin{equation}
\col(F):=\col(\etree) :=\etree\ot\etree.
\mlabel{eq:eqcobi0}
\end{equation}
If $\bre(F)=1$, then $F\in \calg$ and we define
\begin{equation}
\col(F):=\sum_{V_{1}\sqcup V_{2}=V(F)} F[V_{1}]\ot F[V_{2}],
\mlabel{eq:eqcobi}
\end{equation}
where $F[V_{i}]$ are induced subgraphs of $F$ by vertex sets $V_i$ with $i=1,2$.
Note
$$F[V_{1}], F[V_{2}]\in M(\calg)\,\text{ and }\,  \col(F)\in H(\calg) \ot H(\calg).$$
If $\bre(F)>1$, then $F=\Gamma_1\cdots\Gamma_n$ for some $n\geq 2$ and $\Gamma_1, \ldots, \Gamma_n\in \calg$, and define
\begin{equation}
\col(F):=\col(\Gamma_1\cdots\Gamma_n):=\col(\Gamma_1)\cdots\col(\Gamma_n).
\mlabel{eq:eqcobi2}
\end{equation}
Again note that $$\col(F)\in H(\calg) \ot H(\calg)\,\text{ by } \col(\Gamma_i)\in H(\calg) \ot H(\calg)\,\text{ for each } i=1, \cdots, n.$$

Let us give an example for better insight into $\col$.
\begin{exam} Let $G=\tddeuxx{$\tiny 1$}{$\tiny 2$}$. Then $\tddeuxx{$\tiny 1$}{$\tiny 2$}$ and
$\tddeuxx{$\tiny 1$}{$\tiny 2$}\tddeuxx{$\tiny 1$}{$\tiny 2$}$ are in $M(\calg)$, and
\begin{align*}
\col\left(\tddeuxx{$\tiny 1$}{$\tiny 2$}\right)= &\ 1\ot\tddeuxx{$\tiny 1$}{$\tiny 2$}+\tdun{1}\ot\tdun{2}+\tdun{2}\ot\tdun{1}+\tddeuxx{$\tiny 1$}{$\tiny 2$}\ot 1,\\
\col\left(\tddeuxx{$\tiny 1$}{$\tiny 2$}\tddeuxx{$\tiny 1$}{$\tiny 2$}\right)=&\ \col\left(\tddeuxx{$\tiny 1$}{$\tiny 2$}\right)\col\left(\tddeuxx{$\tiny 1$}{$\tiny 2$}\right)\\
=&\ 1\ot\tddeuxx{$\tiny 1$}{$\tiny 2$}\tddeuxx{$\tiny 1$}{$\tiny 2$}+\tdun{1}\ot\tddeuxx{$\tiny 1$}{$\tiny 2$}\tdun{2}+\tdun{2}\ot\tddeuxx{$\tiny 1$}{$\tiny 2$}\tdun{1}+\tddeuxx{$\tiny 1$}{$\tiny 2$}\ot\tddeuxx{$\tiny 1$}{$\tiny 2$}\\
&+\tdun{1}\ot\tdun{2}\tddeuxx{$\tiny 1$}{$\tiny 2$}+\tdun{1}\tdun{1}\ot\tdun{2}\tdun{2}+\tdun{1}\tdun{2}\ot\tdun{2}\tdun{1}+\tdun{1}\tddeuxx{$\tiny 1$}{$\tiny 2$}\ot\tdun{2}\\
&+\tdun{2}\ot\tdun{1}\tddeuxx{$\tiny 1$}{$\tiny 2$}+\tdun{2}\tdun{1}\ot\tdun{1}\tdun{2}+\tdun{2}\tdun{2}\ot\tdun{1}\tdun{1}+\tdun{2}\tddeuxx{$\tiny 1$}{$\tiny 2$}\ot\tdun{1}\\
&+\tddeuxx{$\tiny 1$}{$\tiny 2$}\ot\tddeuxx{$\tiny 1$}{$\tiny 2$}+\tddeuxx{$\tiny 1$}{$\tiny 2$}\tdun{1}\ot\tdun{2}+\tddeuxx{$\tiny 1$}{$\tiny 2$}\tdun{2}\ot\tdun{1}+\tddeuxx{$\tiny 1$}{$\tiny 2$}\tddeuxx{$\tiny 1$}{$\tiny 2$}\ot 1.
\end{align*}
\mlabel{ex:ex2.1}
\end{exam}

\subsection{A combinatorial description of the coproduct}\mlabel{subs:combdes}
This subsection is devoted to a combinatorial description of the coproduct $\col$ given
in Subsection~\mref{subs:sub2}.

Let $G$ be a graph and $F=\Gamma_1\cdots\Gamma_n\in M(\calg)$ with $\Gamma_1,\cdots,\Gamma_n\in\calg$. Note that $V(F)$ is a multiset. For example, let $G=\,\,\squares{$2$}{$1$}{$3$}{$4$}\,\,\,$, $\Gamma=\tddeuxx{$\tiny 1$}{$\tiny 2$}\in \calg$. Take $F:=\Gamma\Gamma=\tddeuxx{$\tiny 1$}{$\tiny 2$}\tddeuxx{$\tiny 1$}{$\tiny 2$}\in M(\calg)$. Then $V(F)=\{1,2,1,2\}$. Now we define
\begin{equation}
\col(F):=\sum_{U\uplus V=V(F)}F[U]\ot F[V],
\mlabel{eq:comni}
\end{equation}
where $U\uplus V$ means that $U$ and $V$ can be expressed as
$$U=U_1\cup\cdots\cup U_n\,\,\text{ and }\,\,V=V_1\cup \cdots \cup V_n,$$
with
$$U_i\sqcup V_i=V(\Gamma_i)\,\,\text{ for }\,\,i=1,\cdots,n.$$
 Further, the coproduct
 $\col$ can be extended to $H(\calg) = \bfk M(\calg)$ by linearity.
Let us compute some examples for better understanding of Eq.~(\mref{eq:comni}).

\begin{exam}
Let $G=\para{$1$}{$2$}\tddeuxx{$3$}{$4$}$ and  $\Gamma_2\Gamma_1=\tddeuxx{$3$}{$4$}\para{$1$}{$2$}\in M(\calg)$ with
$\Gamma_1=\para{$1$}{$2$}$\,\,,\, $\Gamma_2=\tddeuxx{$3$}{$4$}$. Then $$V(\Gamma_1)=\{1,2\}\uplus\varnothing=\{1\}\uplus\{2\}=\{2\}\uplus\{1\}=\varnothing\uplus\{1,2\},$$ $$V(\Gamma_2)=\{3,4\}\uplus\varnothing=\{3\}\uplus\{4\}=\{4\}\uplus\{3\}=\varnothing\uplus\{3,4\},$$
and so
$$\col(\Gamma_1)=\col(\,\,\para{$1$}{$2$}\,\,)=\para{$1$}{$2$}\,\,\ot \etree+\etree\ot\para{$1$}{$2$}+\tdun{$1$}\ot\tdun{$2$}+\tdun{$2$}\ot\tdun{$1$},$$
$$\col(\Gamma_2)=\col(\tddeuxx{$3$}{$4$})=\tddeuxx{$3$}{$4$}\ot \etree+\etree\ot\tddeuxx{$3$}{$4$}+\tdun{$3$}\ot\tdun{$4$}+\tdun{$4$}\ot\tdun{$3$}.$$
Further
\begin{align*}
\col(\tddeuxx{$3$}{$4$}\,\,\para{$1$}{$2$}\,\,)=&\tddeuxx{$3$}{$4$}\,\,\para{$1$}{$2$}\ot \etree+\tddeuxx{$3$}{$4$}\tdun{$1$}\ot\tdun{$2$}+\tddeuxx{$3$}{$4$}\tdun{$2$}\ot\tdun{$1$}+\tddeuxx{$3$}{$4$}\ot\,\,\para{$1$}{$2$}\\
&+\tdun{$3$}\,\,\para{$1$}{$2$}\ot\tdun{$4$}+\tdun{$3$}\tdun{$1$}\ot\tdun{$4$}\tdun{$2$}+\tdun{$3$}\tdun{$2$}\ot\tdun{$4$}\tdun{$1$}
+\tdun{$3$}\ot\tdun{$4$}\,\,\para{$1$}{$2$}\\
&+\tdun{$4$}\,\,\para{$1$}{$2$}\ot\tdun{$3$}+\tdun{$4$}\tdun{$1$}\ot\tdun{$3$}\tdun{$2$}+\tdun{$4$}\tdun{$2$}\ot\tdun{$3$}\tdun{$1$}
+\tdun{$4$}\ot\tdun{$3$}\,\,\para{$1$}{$2$}\\
&+\para{$1$}{$2$}\,\,\ot\tddeuxx{$3$}{$4$}+\tdun{$1$}\ot\tddeuxx{$3$}{$4$}\tdun{$2$}+\tdun{$2$}\ot\tddeuxx{$3$}{$4$}\tdun{$1$}+\etree\ot
\tddeuxx{$3$}{$4$}\,\,\para{$1$}{$2$}\\
=&\col(\tddeuxx{$3$}{$4$})\col(\para{$1$}{$2$}\,\,).
\end{align*}
\end{exam}

\begin{exam}
Consider $G=\tddeuxx{$3$}{$4$}\tdtroisdeux{$1$}{$2$}{$5$}$ and $F=\tddeuxx{$1$}{$2$}\tdtroisdeux{$1$}{$2$}{$5$}\in M(\calg)$. Then
\begin{align*}
\col(\tddeuxx{$1$}{$2$})=&\ \tddeuxx{$1$}{$2$}\ot \etree+\etree\ot\tddeuxx{$1$}{$2$}+\tdun{$1$}\ot\tdun{$2$}+\tdun{$2$}\ot\tdun{$1$},\\
\col\left(\tdtroisdeux{$1$}{$2$}{$5$}\right)=&\ \tdtroisdeux{$1$}{$2$}{$5$}\ot \etree+\etree\ot\tdtroisdeux{$1$}{$2$}{$5$}+\tddeuxx{$1$}{$2$}\ot\tdun{$5$}+\tddeuxx{$2$}{$5$}\ot\tdun{$1$}+\tdun{$1$}\ot
\tddeuxx{$2$}{$5$}+\tdun{$5$}\ot\tddeuxx{$1$}{$2$}+\tdun{$2$}\ot\tdun{$1$}\tdun{$5$}+\tdun{$1$}\tdun{$5$}\ot\tdun{$2$},
\end{align*}
 and
\begin{align*}
\col\left(\tddeuxx{$1$}{$2$}\tdtroisdeux{$1$}{$2$}{$5$}\right)=&\ \tddeuxx{$1$}{$2$}\tdtroisdeux{$1$}{$2$}{$5$}\ot \etree+\tdtroisdeux{$1$}{$2$}{$5$}\ot\tddeuxx{$1$}{$2$}+\tdun{$1$}\tdtroisdeux{$1$}{$2$}{$5$}\ot\tdun{$2$}+\tdun{$2$}\tdtroisdeux{$1$}{$2$}{$5$}\ot\tdun{$1$}\\
&+\tddeuxx{$1$}{$2$}\ot\tdtroisdeux{$1$}{$2$}{$5$}+\etree\ot\tddeuxx{$1$}{$2$}\tdtroisdeux{$1$}{$2$}{$5$}+\tdun{$1$}\ot\tdun{$2$}\tdtroisdeux{$1$}{$2$}{$5$}
+\tdun{$2$}\ot\tdun{$1$}\tdtroisdeux{$1$}{$2$}{$5$}\\
&+\tddeuxx{$1$}{$2$}\tddeuxx{$1$}{$2$}\ot\tdun{$5$}+\tddeuxx{$1$}{$2$}\ot\tddeuxx{$1$}{$2$}\tdun{$5$}+\tdun{$1$}\tddeuxx{$1$}{$2$}\ot\tdun{$2$}\tdun{$5$}
+\tdun{$2$}\tddeuxx{$1$}{$2$}\ot\tdun{$1$}\tdun{$5$}\\
&+\tddeuxx{$1$}{$2$}\tddeuxx{$2$}{$5$}\ot\tdun{$1$}+\tddeuxx{$2$}{$5$}\ot\tddeuxx{$1$}{$2$}\tdun{$1$}+\tdun{$1$}\tddeuxx{$2$}{$5$}\ot\tdun{$2$}\tdun{$1$}
+\tdun{$2$}\tddeuxx{$2$}{$5$}\ot\tdun{$1$}\tdun{$1$}\\
&+\tddeuxx{$1$}{$2$}\tdun{$1$}\ot\tddeuxx{$2$}{$5$}+\tdun{$1$}\ot\tddeuxx{$1$}{$2$}\tddeuxx{$2$}{$5$}+\tdun{$1$}\tdun{$1$}\ot\tdun{$2$}\tddeuxx{$2$}{$5$}
+\tdun{$2$}\tdun{$1$}\ot\tdun{$1$}\tddeuxx{$2$}{$5$}\\
&+\tddeuxx{$1$}{$2$}\tdun{$2$}\ot\tdun{$1$}\tdun{$5$}+\tdun{$2$}\ot\tddeuxx{$1$}{$2$}\tdun{$1$}\tdun{$5$}+\tdun{$1$}\tdun{$2$}\ot
\tdun{$2$}\tdun{$1$}\tdun{$5$}+\tdun{$2$}\tdun{$2$}\ot\tdun{$1$}\tdun{$1$}\tdun{$5$}\\
&+\tddeuxx{$1$}{$2$}\tdun{$5$}\ot\tddeuxx{$1$}{$2$}+\tdun{$5$}\ot\tddeuxx{$1$}{$2$}\tddeuxx{$1$}{$2$}+\tdun{$1$}\tdun{$5$}\ot\tdun{$2$}\tddeuxx{$1$}{$2$}
+\tdun{$2$}\tdun{$5$}\ot\tdun{$1$}\tddeuxx{$1$}{$2$}\\
&+\tddeuxx{$1$}{$2$}\tdun{$1$}\tdun{$5$}\ot \tdun{$2$}+\tdun{$1$}\tdun{$5$}\ot\tddeuxx{$1$}{$2$}\tdun{$2$}+\tdun{$1$}\tdun{$1$}\tdun{$5$}\ot\tdun{$2$}\tdun{$2$}+\tdun{$2$}\tdun{$1$}\tdun{$5$}\ot\tdun{$1$}\tdun{$2$}\\
=&\ \col(\tddeuxx{$1$}{$2$})\col\left(\tdtroisdeux{$1$}{$2$}{$5$}\right).
\end{align*}
\end{exam}

Now we prove that the combinatorial definition of $\col$ in Eq.~(\mref{eq:comni}) is the same as the one given in Subsection~\mref{subs:sub2}.
For this, it suffices to show that $\col$ in Eq.~(\mref{eq:comni}) satisfies Eqs.~(\mref{eq:eqcobi0})---(\mref{eq:eqcobi2}).
For the cases of Eqs.~(\mref{eq:eqcobi0}) and~(\mref{eq:eqcobi}), it follows directly from the definition of $\col$ in Eq.~(\mref{eq:comni}).
For the case of Eq.~(\mref{eq:eqcobi2}), we have

\begin{lemma}
Let $\col$ be given in Eq.~(\mref{eq:comni}) and $\Gamma_1,\cdots,\Gamma_n\in\calg$ with $n\geq 2$. Then
$$\col(\Gamma_1\cdots\Gamma_n)=\col(\Gamma_1)\cdots\col(\Gamma_n).$$
\mlabel{le:lecomo}
\end{lemma}

\begin{proof}
We have
\begin{align*}
\col(\Gamma_1\cdots\Gamma_n)&=\,\sum_{U\uplus V=V(\Gamma_1\cdots\Gamma_n)}(\Gamma_1\cdots\Gamma_n)[U]\ot(\Gamma_1\cdots\Gamma_n)[V]\quad(\text{by Eq}.~(\mref{eq:comni}))\\
&=\sum_{\mbox{\tiny$\begin{array}{c}U_i\sqcup V_i=V(\Gamma_i)\\i=1,\cdots,n\\\end{array}$}}(\Gamma_1\cdots\Gamma_n)[U_1\cup\cdots\cup U_n]\ot(\Gamma_1\cdots\Gamma_n)[V_1\cup\cdots\cup V_n]\\
&=\sum_{\mbox{\tiny$\begin{array}{c}U_i\sqcup V_i=V(\Gamma_i)\\i=1,\cdots,n\\\end{array}$}}\Gamma_1[U_1]\cdots\Gamma_n[U_n]\ot\Gamma_1[V_1]\cdots\Gamma_n[V_n]\\
&=\left(\sum_{U_1\sqcup V_1=V(\Gamma_1)}\Gamma_1[U_1]\ot\Gamma_1[V_1]\right)\cdots\left(\sum_{U_n\sqcup V_n=V(\Gamma_n)}\Gamma_n[U_n]\ot\Gamma_n[V_n]\right)\\
&=\col(\Gamma_1)\cdots\col(\Gamma_n)\quad(\text{by Eq}.~(\mref{eq:comni})).
\end{align*}
This completes the proof.
\end{proof}

So we conclude

\begin{prop}
The $\col$ given in Eq.~(\mref{eq:comni}) coincides with the $\col$ given in Subsection~\mref{subs:sub2}.
\end{prop}

\begin{remark}
The coproduct $\col$ in Eq.~(\mref{eq:comni}) is cocommutative.
\end{remark}

\subsection{A coalgebraic structure on subgraphs of a given graph} \mlabel{ssec:coalg}
In this subsection, we obtain a coalgebra structure on subgraphs of a given graph $G$.
Firstly, we define a linear map $\epl: H(\calg)\rightarrow\bfk$ by taking
\begin{equation}
\epl(F) :=
\left\{
\begin{array}{ll}
1_\bfk, & \text{ if } F = \etree, \\
0, & \text{ if } F \neq \etree.\\
\end{array}
\right.
 \mlabel{eq:biacou}
\end{equation}

\begin{lemma}
Let $G$ be a graph. Then the triple $(H(\calg),\,\col,\,\epl)$ is a coalgebra.
\mlabel{le:lecounicoal}
\end{lemma}

\begin{proof}
We first show the coassociativity
\begin{equation*}
(\id\otimes \col)\col(F)=(\col\otimes \id)\col(F)\,\text{ for } F\in H(\calg).
\end{equation*}
By linearity, it suffices to consider basis elements $F\in M(\calg)$.
We have
\begin{align*}
(\id\otimes \col)\col(F)=&\,(\id\ot\col)\left(\sum_{U\uplus V=V(F)}F[U]\ot F[V]\right)\quad(\text{by Eq}.~(\mref{eq:comni}))\\
=&\,\sum_{U\uplus V=V(F)}F[U]\ot\col(F[V])\\
=&\,\sum_{U\uplus V=V(F)}F[U]\ot \left(\sum_{V'\uplus V''=V(F[V])}(F[V])[V']\ot (F[V])[V'']\right)\quad(\text{by Eq}.~(\mref{eq:comni}))\\
=&\,\sum_{U\uplus V=V(F)}F[U]\ot \left(\sum_{V'\uplus V''=V(F[V])}F[V']\ot F[V'']\right)\\
=&\,\sum_{U\uplus V'\uplus V''=V(F)}F[U]\ot F[V']\ot F[V'']\quad(\text{by } V(F[V])=V)\\
=&\,\sum_{U\uplus V'\uplus V''=V(F)}(F[U\uplus V'])[U]\ot (F[U\uplus V'])[V']\ot F[V'']\\
=&\,\sum_{W\uplus V''=V(F)}\left(\sum_{U\uplus V'=W}F[W][U]\ot F[W][V']\right)\ot F[V'']\\
=&\,\sum_{W\uplus V''=V(F)}\left(\sum_{U\uplus V'=V(F[W])}F[W][U]\ot F[W][V']\right)\ot F[V'']\quad(\text{by } V(F[W])=W)\\
=&\,\sum_{W\uplus V''=V(F)}\col(F[W])\ot F[V'']\quad(\text{by }U\uplus V'=W)\\
=&\,(\col\ot\id)\left(\sum_{W\uplus V''=V(F)}F[W]\ot F[V'']\right)\\
=&\,(\col\ot\id)\col(F).
\end{align*}

Thus $\col$ is coassociative.
Next we prove the counicity of $\epl$. For  $ F\in M(\calg)$,
\begin{align*}
(\epl\ot \id)\col(F)=&\ (\epl\ot\id)\left(\sum_{U\uplus V=V(F)}F[U]\ot F[V]\right)\quad(\text{by Eq.~}(\mref{eq:comni}))\\
=&\ (\epl\ot\id)\left(\etree\ot F+F\ot \etree+\sum_{\mbox{\tiny$\begin{array}{c}U\uplus V=V(F)\\ U,V\neq V(F)\\\end{array}$}}F[U]\ot F[V]\right)\\
=&\ \epl(\etree)\ot F+\epl(F)\ot \etree+\sum_{\mbox{\tiny$\begin{array}{c}U\uplus V=V(F)\\ U,V\neq V(F)\\\end{array}$}} \epl(F[U])\ot F[V]\\
=& \ 1_\bfk\ot F = \ \beta_{l}(F)\quad(\text{by Eq.~}(\mref{eq:biacou})).
\end{align*}
With the same argument, we can show  $(\id\ot\epl)\col=\beta_r$. Here $\beta_l$ and $\beta_r$
are the isomorphisms defined in Definition~\mref{de:decoa}.
This completes the proof.
\end{proof}

\begin{coro}
Let $G$ be a graph and $G'$ be a subgraph of graph $G$. Then the triple $(H(\calg'),\,\col,\,\epl)$ is a subcoalgebra of the coalgebra $(H(\calg),\,\col,\,\epl)$.
\mlabel{co:coca}
\end{coro}
\begin{proof}
It follows from the Lemma~\mref{le:lecounicoal} and Eqs.~(\mref{eq:comni}) and~(\mref{eq:biacou}).
\end{proof}

\begin{coro}
Let $G$ be a graph. Then $(H(\bullet),\,\col,\,\epl)$ is a simple subcoalgebra, for any $\bullet\in V(G)$.
\mlabel{co:cosca}
\end{coro}
\begin{proof}
It follows directly from the Corollary~\mref{co:coca} and  Eq.~(\mref{eq:comni}).
\end{proof}

\subsection{A bialgebraic structure on subgraphs of a given graph}
Combing the algebraic structure  obtained in Subsection~\mref{subs:sub1} and the coalgebraic structure obtained in Subsection~\mref{ssec:coalg},
we build a bialgebraic structure on top of subgraphs of a given graph. Let us review the concept of bialgebras.

\begin{defn}
\begin{enumerate}
\item\cite[p. 49]{Gub}
Let $(H,\,m_H,\,u_H)$ and $(L,\,m_L,\,u_L)$ be two algebras. A map $f:H\rightarrow L$ is called an {\bf algebra morphism} if
$$m_L\circ (f(x)\ot f(y))=f\circ m_H(x\ot y)\,\,\,\text{ and }\,\,\,u_L(k)=f\circ u_H(k)\,\,\text{for all}\,\,x,y\in H,k\in\bfk.$$
\item\cite[p. 51]{Gub}
A {\bf bialgebra} is a quintuple $(H,\,m,\,u,\,\Delta,\,\varepsilon)$, where $(H,\,m,\,u)$ is an algebra and $(H,\,\Delta,\,\varepsilon)$ is a coalgebra such that $\Delta: H\rightarrow H\ot H$ and $\varepsilon: H\rightarrow \bfk$ are morphisms of algebras.
\end{enumerate}
\end{defn}

Now, we arrive at our main result in this subsection.

\begin{theorem}
Let $G$ be a graph. Then the quintuple $(H(\calg),\,\mul,\,\ul,\,\col,\,\epl)$ is a bialgebra.
\mlabel{th:bialg}
\end{theorem}

\begin{proof}
In Subsection~\mref{subs:sub1}, we know that the triple $(H(\calg), \,\mul,\ul)$ is an algebra. The triple $(H(\calg),\, \col,\,\epl )$ is a coalgebra by Lemma~\mref{le:lecounicoal}. Further, the coproduct $\col$ is an algebra morphism by Eq.~(\mref{eq:eqcobi2}).
We are left to prove that $\epl$ is an algebra morphism.
For this, let $F_{1}, F_2\in M(\calg)$. Then by Eq.~(\mref{eq:biacou})
\begin{equation*}
\epl(F_1F_2) =
\left\{
\begin{array}{ll}
1_\bfk, & \text{ if } F_1 \text{ and } F_{2} \text{ are  empty graphs},\\
0, &  \text{ others }, \\
\end{array}
\right .
\end{equation*}
and
\begin{equation*}
\epl(F_1)\epl(F_2) =
\left\{
\begin{array}{ll}
1_\bfk, & \text{ if } F_1 \text{ and } F_{2} \text{ are  empty graphs},\\
0, & \text{ others}. \\
\end{array}
\right .
\end{equation*}
Thus, $\epl(F_1F_2)=\epl(F_1)\epl(F_2)$ and $\epl$ is an algebra morphism. This completes the proof.
\end{proof}

\section{A Hopf algebra on subgraphs of a given graph}
\mlabel{sec:hopfalg}
In this section, we construct a Hopf algebraic structure on subgraphs of a given graph $G$.
\subsection{A Hopf algebraic structure on subgraphs of a graph}
Let us review some concepts needed later.

\begin{defn}\cite[p.~61]{Abe}
Let  $(C,\,\Delta,\,\varepsilon)$ be a coalgebra and $(A,\,m,\,u)$ be an algebra. For $f,g\in \Hom(C,A)$,
$$f\ast g:=m\circ(f\ot g)\circ \Delta$$
is said to be the  {\bf convolution} of $f$ and $g$.
\end{defn}

\begin{defn}\cite[p.~61]{Abe}
Let $(H,\,m,\,u,\,\Delta,\,\varepsilon)$ be a bialgebra. Then the triple $(\Hom(H,H),\,\ast,\,u\circ\varepsilon)$ is a $\bfk$-algebra. A linear endomorphism $S$ of $H$ is called an {\bf antipode} for $H$ if it is the inverse of $\id_H$ under the convolution product:
$$S\ast \id_H=\id_H\ast S=u\circ\varepsilon.$$
A bialgebra with an antipode is called a {\bf Hopf algebra}.
\end{defn}

\begin{defn}\cite[Definition~2.3.1]{Gub}
A bialgebra $(H,\,m,\,u,\,\Delta,\,\varepsilon)$ is called a {\bf graded bialgebra} if there are $\bfk$-submodules $H^{(n)}, n\geq 0$, of $H$ such that
\begin{enumerate}
\item
$H=\bigoplus\limits^{\infty}_{n=0}H^{(n)}$;
\vspace{0.2cm}
\item
$H^{(p)}H^{(q)}\subseteq H^{(p+q)},\,\, p, q\geq 0;$
\vspace{0.2cm}
\item
$\Delta(H^{(n)})\subseteq\bigoplus\limits_{p+q=n}H^{(p)}\ot H^{(q)},\,\, n\geq 0,\,\,0\leq p,q\leq n.$
\end{enumerate}
Elements of $H^{(n)}$ are said to have degree $n$. The $H$ is called {\bf connected} if $H^{(0)}=\bfk$ and $\ker \varepsilon=\bigoplus\limits_{n\geq 1}H^{(n)}$.
\end{defn}
\begin{lemma}\cite{CK, Fo3, Gub}
Any connected graded bialgebra $H$ is a Hopf algebra. The antipode $S$ is given by:
$$S(x)=\sum_{k\geq 0}(e-\id_{H})^{\ast k}(x).$$
It is also defined by $S(1_{H})=1_H$ and recursively by any of the two formulas:
\begin{align*}
&S(x)=-x-\sum_{(x)}S(x')x'',\\
&S(x)=-x-\sum_{(x)}x'S(x'')\,\,\text{ for }\,\,x\in\ker \varepsilon.
\end{align*}
Here $x'$ and $x''$ are from $$\Delta(x)=x\ot 1+1\ot x+\sum_{(x)}x'\ot x''.$$
\mlabel{le:conhop}
\end{lemma}

We proceed to prove that $H(\calg)$ is a connected graded bialgebra. For this, denoted by
$$H(\calg)^{(n)}:=\bfk\left\{F\in  M(\calg)\medmid \left|V(F)\right|= n\right\}\,\, \text{ for } n\geq 0,$$
where $V(F)$ is the multiset of vertices of $F$. Now we arrive at one of our main results in this section.
\begin{theorem}
Let $G$ be a graph. Then
\begin{enumerate}
\item
The quintuple $(H(\calg),\,\mul,\,\ul,\,\col,\,\epl)$ is a connected graded bialgebra.\mlabel{it:ita}
\item
The quintuple $(H(\calg),\,\mul,\,\ul,\,\col,\,\epl)$ is a Hopf algebra with antipode $S=\sum\limits_{k\geq 0}(e-\id_{H})^{\ast k}$.\mlabel{it:itb}
\end{enumerate}
\mlabel{th:thcg}
\end{theorem}

\begin{proof}
\noindent{\bf (a).}
The quintuple $(H(\calg),\,\mul,\,\ul,\,\col,\,\epl)$ is a bialgebra by Theorem~\mref{th:bialg}.
By the definition of $H(\calg)$, we obtain
$$H(\calg)=\bigoplus\limits^{\infty}_{n=0}H(\calg)^{(n)}.$$
 Let $F_1\in H(\calg)^{(p)}$ and $F_2\in H(\calg)^{(q)}$ with $p, q\geq 0$. Then
 $$\left|V(F_1)\right|=p,\,  \left|V(F_2)\right|=q\,\,\,\text{ and }\,\,\,\left|V(F_1F_2)\right|=pq,$$ whence

$$H(\calg)^{(p)}H(\calg)^{(q)}\subseteq H(\calg)^{(p+q)}.$$
By Eq.~(\mref{eq:comni}),
$$\col(H(\calg)^{(n)})\subseteq \bigoplus\limits_{p+q=n}H(\calg)^{(p)}\ot H(\calg)^{(q)}.$$
Thus, the quintuple $(H(\calg),\,\mul,\,\ul,\,\col,\,\epl)$ is a graded bialgebra. Further,
$$H(\calg)^{(0)}=\bfk\{\etree\}=\bfk\,\text{ and }\, \ker\varepsilon=\bigoplus\limits_{n\geq 1}H(\calg)^{(n)}\,\text{ by Eq.~(\mref{eq:biacou})}.$$
Therefore, the quintuple $(H(\calg),\,\mul,\,\ul,\,\col,\,\epl)$ is a connected graded bialgebra.

\noindent{\bf (b).}
It follows from Lemma~\mref{le:conhop} and Item~(\mref{it:ita}).
\end{proof}

\subsection{The dual Hopf algebra}
In this subsection, we consider the dual Hopf algebra of $(H(\calg),\,\mul, \ul,\,\col,\,\epl)$. The following elementary result is
fundamental.

\begin{lemma}~\cite{Fo3}
Let $H=\bigoplus\limits^{\infty}_{n=0}H^{(n)}$ be a graded bialgebra and $H^{(n)}$ be finite dimensional. Then
\begin{enumerate}
\item
The graded dual $H^{\ast}$ is $\bigoplus\limits^{\infty}_{n=0}\left(H^{(n)}\right)^{\ast}$. Note that $H^{\ast}$ is also a graded bialgebra, and
$H^{\ast\ast}\simeq H$.
\item $H\ot H$ is also a graded bialgebra with $(H\ot H)^{(n)}=\sum\limits_{i=0}^n H^{(i)}\ot H^{(n-i)}$ for all $n\in \mathbb{N}$. Moreover, $(H\ot H)^{\ast}\simeq H^\ast\ot H^\ast$.
\end{enumerate}
\mlabel{le:ledu}
\end{lemma}

The quintuple $(H(\calg),\,\mul,\,\ul,\,\col,\,\epl)$ is a connected graded bialgebra by Theorem~\mref{th:thcg} and $H(\calg)^{(n)}$ is finite dimensional. Then its connected graded dual inherits also a graded Hopf algebra by Lemma~\mref{le:ledu}, denoted by $(H(\calg)^\ast,\,\col^\ast,\,\epl^{\ast},\,\mul^\ast,\,\ul^\ast)$.

For each $F\in M(\calg)$, we define
\begin{align}
Z_F:H(\calg)\longrightarrow\bfk,\,\, F'\longrightarrow\delta_{F,F'}\,\,\text{ for }\,\,F'\in M(\calg),
\mlabel{eq:eqdu}
\end{align}
 where $\delta_{F,F'}$ is Kronecker function. As $\left\{Z_F \medmid F\in M(\calg)^{(n)}\right\}$ is a basis of $(H(\calg)^{(n)})^{\ast}$, $\left\{Z_F \medmid F\in M(\calg)\right\}$ is a basis of $H(\calg)^\ast$.

 The following result is a combinatorial description of the  multiplication in $H(\calg)^\ast$.
 \begin{theorem}
 Let $G$ be graph and $F_1,\,F_2\in M(\calg)$.  The product of $Z_{F_1}$ and $Z_{F_2}$ is given by
 \begin{equation}
 Z_{F_1}Z_{F_2}=\sum_{F\in M(\calg)}n(F_1,F_2;F)Z_{F},
 \end{equation}
 where $n(F_1,F_2;F)$ is the number of $F_1\ot F_2$ in  $\col(F)$.
 \mlabel{le:ledup}
 \end{theorem}
 \begin{proof}
 Let $G$ be a graph. We consider the basis elements of $H(\calg)^\ast$. For any $F_1,\,F_2\in M(\calg)$, we suppose
 $$Z_{F_1}Z_{F_2}=\sum_{F'\in M(\calg)}c_{F_1,F_2}^{F'}Z_{F'}.$$
For any $F\in M(\calg)$, we have
$$\sum_{F'\in M(\calg)}c_{F_1,F_2}^{F'}Z_{F'}(F)=c_{F_1,F_2}^{F}Z_{F}(F)=c_{F_1,F_2}^{F},$$
and
 \begin{align*}
 Z_{F_1}Z_{F_2}(F)&=\,\col^\ast(Z_{F_1}\ot Z_{F_2})(F)=\,(Z_{F_1}\ot Z_{F_2})(\col(F))\\
 &=\,(Z_{F_1}\ot Z_{F_2})\left(\sum_{U\uplus V=V(F)}F[U]\ot F[V]\right)\\
 &=\,\sum_{U\uplus V=V(F)}Z_{F_1}(F[U])\ot Z_{F_2}(F[V])\\
 &=\,\sum_{U\uplus V=V(F)}\delta_{F_1,F[U]}\ot \delta_{F_2,F[V]}\\
 &=\,n(F_1,F_2;F)\quad(\text{by Eq.~(\mref{eq:comni})}).
 \end{align*}
Thus  $c_{F_1,F_2}^{F}=n(F_1,F_2;F)$. This completes the proof.
 \end{proof}
 \begin{exam}Let $G$ be a graph in Item~(\mref{ex:itex3}) of Example~\mref{ex:exs}. We consider the product of $H(\calg)^\ast$.
 \begin{enumerate}
 \item
 Let $F_1=\tdun{1}$ and $F_2=\tddeuxx{$\tiny 2$}{$\tiny 1$}$. Then
 $$Z_{\tdun{1}}Z_{\tddeuxx{$\tiny 2$}{$\tiny 1$}}=Z_{\tdun{1}\tddeuxx{$\tiny 2$}{$\tiny 1$}}+Z_{\tddeuxx{$\tiny 2$}{$\tiny 1$}\tdun{1}}.$$
 \item
 Let $F_1=\tdun{3}$ and $F_2=\tddeuxx{$\tiny 2$}{$\tiny 1$}$. We have
 $$Z_{\tdun{3}}Z_{\tddeuxx{$\tiny 2$}{$\tiny 1$}}=Z_{\tdun{3}\tddeuxx{$\tiny 2$}{$\tiny 1$}}+Z_{\tddeuxx{$\tiny 2$}{$\tiny 1$}\tdun{3}}+Z_{\,\,\trpoinb{$2$}{$1$}{$3$}}\,\,.$$
 \end{enumerate}
 \end{exam}

 The coproduct $\mul^{\ast}$ of $H(\calg)^\ast$ is given by
 \begin{lemma}
 Let $G$ be graph and $\Gamma_1\cdots\Gamma_n\in M(\calg)$. The coproduct of $Z_{\Gamma_1\cdots\Gamma_n}\in M(\calg)^\ast$ is given by
 \begin{equation}
\mul^{\ast}( Z_{\Gamma_1\cdots\Gamma_n})=\sum_{0\leq i\leq n}Z_{\Gamma_1\cdots\Gamma_i}\ot Z_{\Gamma_{i+1}\cdots\Gamma_n},
 \end{equation}
 with the convention that $Z_{\Gamma_1\Gamma_0}=\etree$ and $Z_{\Gamma_{n+1}\Gamma_n}=\etree$.
 \mlabel{le:leducp}
 \end{lemma}
 \begin{proof}
 We consider the basis elements of $H(\calg)^{\ast}$.
 Suppose $$m^\ast(Z_{\Gamma_1\cdots\Gamma_n})=\sum_{F',F''\in M(\calg)}c_{F',F''}Z_{F'}\ot Z_{F''}.$$
  For any $F_1, F_2 \in M(\calg)$, we have
 \begin{align*}
 \mul^{\ast}( Z_{\Gamma_1\cdots\Gamma_n})(F_1\ot F_2)=Z_{\Gamma_1\cdots\Gamma_n}(\mul(F_1\ot F_2))=Z_{\Gamma_1\cdots\Gamma_n}(F_1F_2)=\delta_{\Gamma_1\cdots\Gamma_n, F_1F_2},
 \end{align*}
 and
 \begin{align*}
&\left(\sum_{F',F''\in M(\calg)}c_{F',F''}(Z_{F'}\ot Z_{F''})\right)(F_1\ot F_2)\\
=&\,\sum_{F',F''\in M(\calg)}c_{F',F''}Z_{F'}(F_1)\ot Z_{F''}(F_2)\\
=&\,\sum_{F',F''\in M(\calg)}c_{F',F''}\delta_{F',F_1}\ot \delta_{F'',F_2}\\
=&\,c_{F_1,F_2}.
 \end{align*}
 Thus,
$c_{F_1,F_2}=1$ if $\Gamma_1\cdots\Gamma_n=F_1F_2$ and $c_{F_1,F_2}=0$ otherwise. This completes the proof.
 \end{proof}
 \begin{exam}
 Let $G$ be a graph in Item~(\mref{ex:itex3}) of Example~\mref{ex:exs}. We consider the coproduct of $H(\calg)^\ast$.  Let $\Gamma_1=\tdun{1}$ and $\Gamma_2=\tddeuxx{$\tiny 2$}{$\tiny 1$}$. Then
 $$Z_{\tdun{1}\tddeuxx{$\tiny 2$}{$\tiny 1$}}=Z_{\tdun{1}\tddeuxx{$\tiny 2$}{$\tiny 1$}}\ot \etree+\etree\ot Z_{\tdun{1}\tddeuxx{$\tiny 2$}{$\tiny 1$}}+Z_{\tdun{1}}\ot Z_{\tddeuxx{$\tiny 2$}{$\tiny 1$}}.$$
 \end{exam}
\subsection{ The algebra morphisms induced from graph homomorphisms}
In this subsection, we are going to consider algebra morphisms induced by graph homomorphisms.
For this, let us recall the concept of a graph homomorphism.

\begin{defn}\cite[p.~4]{Heb}
Let $G_1$ and $G_2$ be two graphs.
A {\bf homomorphism} of $G_1$ to $G_2$, written as $f: G_1\rightarrow G_2$, is a mapping $f: V(G_1)\rightarrow V(G_2)$ such that $(f(u),f(v))\in E(G_2)$ whenever $(u,v)\in E(G_1)$.
\mlabel{de:ghom}
\end{defn}

Next, we give the main result in this subsection.
Let $f: G_{1}\rightarrow G_{2}$ be a graph homomorphism. Then $f$ can be restricted to any subgraphs of $G_1$, still denoted by $f$.
We define the linear map $H(f): H(\calg_{1})\rightarrow H(\calg_{2})$  given by $$H(f)(\etree_{G_1}):=\etree_{G_{2}}$$ and
\begin{equation}
H(f)(F):=H(f)(\Gamma_1\cdots\Gamma_n):=f(\Gamma_1)\cdots f(\Gamma_n)\,\,\text{ for }\,\,F=\Gamma_1\cdots\Gamma_n\in M(\calg_1)\,\,\text{ with }\,\,n\geq 1.
 \mlabel{eq:demor}
\end{equation}

\begin{theorem}
The map $H(f): H(\calg_{1})\rightarrow H(\calg_{2})$ is an algebra morphism.
\mlabel{th:grhop}
\end{theorem}

\begin{proof}
Let  $$F_1=\Gamma_1\cdots\Gamma_m\,\,\text{ and }\,\, F_2=\Gamma'_1\cdots\Gamma'_n\in M(\calg_1).$$ Then by Eq.~(\mref{eq:demor}),
\begin{align*}
H(f)(F_1F_2)=&\,H(f)(\Gamma_{1}\cdots \Gamma_{m}\Gamma'_{1}\cdots \Gamma'_{n})=\,f(\Gamma_{1})\cdots f(\Gamma_{m})f(\Gamma'_{1})\cdots f(\Gamma'_{n})\\
=&\,H(f)(\Gamma_{1}\cdots \Gamma_{m})\,H(f)(\Gamma'_{1}\cdots \Gamma'_{n})\\
=&\,H(f)(F_1)\,H(f)(F_2).
\end{align*}
Thus $H(f)$ is an algebra morphism. This completes the proof.
\end{proof}

As an application of Theorem~\mref{th:grhop}, we obtain a functor from the graph category to the algebra category.
Let us recall

\begin{defn}~\cite[Definition~7.1]{Hun}
A {\bf category } is a class $\mathfrak{C}$ of objects (denoted $A, B, C, \cdots$) together with
\begin{enumerate}
\item a class of disjoint sets, denoted $\hom(A,B)$, one for each pair of objects in $\mathfrak{C}$; (an element $f$ of $\hom(A,B)$ is called a {\bf morphism} from $A$ to $B$ and is denoted $f: A\rightarrow B$);
\item for each triple $(A, B, C)$ of objects of $\mathfrak{C}$ a function:
$$\hom(B,C)\times\hom(A,B)\rightarrow\hom(A,C);$$
(for morphisms $f: A\rightarrow B$, $g: B\rightarrow C$, this function is written $(g,f)\mapsto g\circ f$ and $g\circ f: A\rightarrow C$ is called the {\bf composite} of $f$ and $g$); all subject to the two axioms:
\begin{enumerate}
\item
Associativity. If $f: A\rightarrow B$, $g: B\rightarrow C$, $h:C\rightarrow D$ are morphisms of $\mathfrak{C}$, then $h\circ(g\circ f)=(h\circ g)\circ f$.
\item
Identity. For each object $B$ of $\mathfrak{C}$ there exists a morphism $1_B: B\rightarrow B$ such that for any $f: A\rightarrow B$, $g: B\rightarrow C$,
$$1_B\circ f=f\,\,\text{ and }\,\,g\circ 1_B=g.$$
\end{enumerate}
\end{enumerate}
\end{defn}

\begin{defn}~\cite[Definition~1.1]{Hun}
Let $\mathfrak{C}$ and $\mathfrak{D}$ be categories. A {\bf covariant functor T} from $\mathfrak{C}$ to $\mathfrak{D}$ (denoted $T:\mathfrak{C}\rightarrow \mathfrak{D}$) is a pair of functions (both denoted by $T$), an object function that assigns to each object $C$ of $\mathfrak{C}$ an object $T(C)$ of $\mathfrak{D}$ and a morphism function which assigns to each morphism $f: C\rightarrow D$ of $\mathfrak{C}$ a morphism
$$T(f): T(C)\rightarrow T(D)$$
of $\mathfrak{D}$, such that
\begin{enumerate}
\item
$T(1_C)=1_{T(C)}$ for every identity morphism $1_C$ of $\mathfrak{C}$;\mlabel{it:1a}
\item
$T(g\circ f)=T(g)\circ T(f)$ for any two morphisms $f,g$ of $\mathfrak{C}$ whose composite $g\circ f$ is defined.\mlabel{it:1b}
\end{enumerate}
\mlabel{de:decof}
\end{defn}
Let $\graphc$ be the class of all graphs, for $G_1, G_2\in\graphc$, $\hom(G_1,G_2)$ is the set of all  graph homomorphisms $f: G_1\rightarrow G_2$. Then $\graphc$ is a category, called the {\bf graph category}.
Let $\hopfc$ be the {\bf $\bfk$-algebra category} whose objects are all $\bfk$-algebra; the  $\hom(A_1,A_2)$ is the set of all $\bfk$-algebra morphisms $\varphi: A_1\rightarrow A_2$.

Now we arrive at out main result of this section.

\begin{theorem}
Let $\graphc$ be the graph category and $\hopfc$ the $\bfk$-algebra category.
Define
$$H(-):\graphc\rightarrow\hopfc,\,\, G\mapsto H(\calg), \,\, f\mapsto H(f),$$
where $G$ is a graph in $\graphc$, $f: G_1\rightarrow G_2$ is a graph homomorphism and $H(f): H(\calg_1)\rightarrow H(\calg_2)$
is the algebra morphism obtained in Theorem~\mref{th:grhop}.
Then $H(-)$ is a covariant functor.
\mlabel{th:thcf}
\end{theorem}
\begin{proof}
 By Theorem~\mref{th:grhop}, it suffices to prove Items~(\mref{it:1a}) and~(\mref{it:1b}) in Definition~\mref{de:decof}.
For Item~(\mref{it:1a}), let $F=\Gamma_1\cdots\Gamma_n\in M(\calg)$. Then by Eq.~(\mref{eq:demor}),
\begin{align*}
H(1_{G})(F)=&\,H(1_{G})(\Gamma_1\cdots\Gamma_n)=\,1_{G}(\Gamma_1)\cdots 1_{G}(\Gamma_n)=\,\Gamma_1\cdots\Gamma_n\\
=&\,1_{H(G)}(\Gamma_1)\cdots 1_{H(G)}(\Gamma_n)=\,1_{H(G)}(\Gamma_1\cdots\Gamma_n)=\,1_{H(G)}(F).
\end{align*}
 For Item~(\mref{it:1b}), let $$f: G_1\rightarrow G_2\,\,\text{ and }\,\,g: G_2\rightarrow G_3,$$ be  graph homomorphisms. For $F=\Gamma_1\cdots\Gamma_n\in M(\calg_1)$,
 \begin{align*}
 H(g)\circ H(f)(F)&=\,H(g)\circ H(f)(\Gamma_1\cdots\Gamma_n)=\,H(g)(f(\Gamma_1)\cdots f(\Gamma_n))\\
&=\,g(f(\Gamma_1))\cdots g(f(\Gamma_n))\\
&=\,(g\circ f)(\Gamma_1)\cdots (g\circ f)(\Gamma_n)\\
&=\,H(g\circ f)(\Gamma_1\cdots\Gamma_n)\\
 &=H(g\circ f)(F),
 \end{align*}
where the third equation employs the fact that the image of a connected subgraph of a graph homomorphism  is a connected subgraph.
Thus  $H(g)\circ H(f)(F)=H(g\circ f)(F)$. This completes the proof.
\end{proof}

\medskip

\noindent {\bf Acknowledgments}: This work was supported by the National Natural Science Foundation
of China (Grant No.\@ 11571155, 11771191, 11861051). We thank the anonymous referee for valuable suggestions helping to improve the paper.

\end{document}